\documentclass{article}

\usepackage[T1]{fontenc}
\usepackage{lmodern}
\usepackage[fleqn,leqno]{mathtools}
\usepackage{amsthm}
\usepackage[letterpaper,centering,textheight=1.2\textheight]{geometry}
\usepackage{microtype}
\usepackage[hidelinks]{hyperref}

\makeatletter
\def\@@footer{\small\hss\thepage\hss}
\def\ps@articleheading{%
\let\@oddfoot\@@footer
\let\@evenfoot\@oddfoot
\def\@oddhead{\small\scshape\hfil\rightmark\hfil}%
\let\@evenhead\@oddhead
\let\@mkboth\@gobbletwo
\let\sectionmark\@gobble
\let\subsectionmark\@gobble
}
\def\ps@plain{%
\let\@oddfoot\@@footer
\let\@evenfoot\@oddfoot  
\def\@oddhead{\small\scshape\@@date\hss}%
\let\@evenhead\@oddhead
\let\@mkboth\@gobbletwo
}
\renewcommand{\thesection}{\S\@arabic\c@section}
\renewcommand{\section}{%
\@startsection{section}{1}{\z@}%
{-3.5ex \@plus -1ex \@minus -.2ex}%
{2.3ex \@plus.2ex}%
{\centering\normalfont\large\scshape}%
}
\renewcommand{\subsection}{%
\@startsection{subsection}{2}{\parindent}%
{\topsep}%
{-.5em}%
{\normalfont\normalsize\itshape}%
}
\renewcommand{\paragraph}{%
\@startsection{paragraph}{4}{\parindent}%
{.5\baselineskip \@plus1ex \@minus.2ex}%
{-.5em}%
{\normalfont\normalsize\itshape}%
}
\let\@afterindentfalse\@afterindenttrue
\@afterindenttrue
\def\@maketitle{%
\global\let\@@date\@date
\newpage
\null
\vskip 2em%
\begin{center}
\let\footnote\thanks
{\Large\@title\par}%
\vskip 1.5em%
{%
\large\lineskip.5em%
\begin{tabular}[t]{c}%
\normalfont\slshape\@author%
\end{tabular}
\par%
}%
\end{center}
\par%
\vskip 1.5em%
}
\renewenvironment{proof}[1][\proofname]{\par
\pushQED{\qed}%
\normalfont \topsep6\p@\@plus6\p@\relax
\trivlist
\item[\hskip\labelsep
\scshape
\hskip\parindent#1\@addpunct{.}]\ignorespaces
}{%
\popQED\endtrivlist\@endpefalse
}
\makeatother

\newtheoremstyle{proclamation}%
{\topsep}{\topsep}{\slshape}{\parindent}{\scshape}{.}{.5em}{}{}
\theoremstyle{proclamation}
\newtheorem{lemma}{Lemma}

\newtheorem*{eichler-congruence}{Eichler's congruence}
\newtheorem*{bch-formula}{Baker--Campbell--Hausdorff formula}

\newcommand{\field}{k}
\newcommand{\powerseries}{A}
\newcommand{\idealXY}{\powerseries^{X,Y}}
\newcommand{\spanXY}{\field^{X,Y}}
\newcommand{\oneover}[1]{\ensuremath{\textstyle\frac1{#1}}}
\DeclarePairedDelimiterX{\lie}[2]{[}{]}{{#1},{#2}}
\DeclarePairedDelimiter{\ilie}{[}{]}
\DeclarePairedDelimiterX{\bch}[2]{\langle}{\rangle}{{#1},{#2}}
\newcommand{\bbch}[2]{\bch[\big]{#1}{#2}}
\newcommand{\homog}[3][n]{{\bch{#2}{#3}}_{#1}}
\newcommand{\homogeven}[3][n]{{\bch{#2}{#3}}^0_{#1}}
\newcommand{\homogodd}[3][n]{{\bch{#2}{#3}}^1_{#1}}
\newcommand{\bhomog}[3][n]{{\bbch{#2}{#3}}_{#1}}
\renewcommand{\deg}[1][n]{\pi_{#1}}
\let\cong\equiv
\let\vec\boldsymbol

\begin{document}

\title{A variation of Eichler's proof of\\
the Baker--Campbell--Hausdorff formula}
\author{Eugene Ha%
\footnote{No A.I.\ whatsoever was used for any aspect of this work.}}
\date{5 August 2026}
\maketitle

Let $\bch XY$ be the formal power series $\log(e^Xe^Y)$ in two noncommutative variables, over a field of characteristic zero. Let $\homog XY$ be the homogeneous term of $\bch XY$ of degree~$n$; thus $\bch XY=\sum_{n\ge1}\homog XY$. For $n=1$, $\homog[1]XY$ is $X+Y$. For $n\ge2$, the Baker--Campbell--Hausdorff formula \cite[Satz~B]{hausdorff} asserts that $\homog XY$ is a \emph{Lie polynomial}, that is, a finite linear combination of iterated commutators of $X$ and~$Y$ (the commutator being the expression $\lie xy=xy-yx$).

Eichler \cite{eichler} established this formula as a surprisingly elementary consequence of the associativity of the product of three exponentials. His proof is short but involved.

In this note we show that the associativity of the product of \emph{four} exponentials enables one to bypass the complications of Eichler's proof to arrive at a shorter proof of the Baker--Campbell--Hausdorff formula, which is both elementary and straightforward. As a corollary of the proof, a recursive scheme to compute $\homog XY$ is derived, which involves no explicit Bernoulli numbers.

\subsection*{Acknowledgements.}

I am grateful to Benoit Jacob for his comments on an early draft, which have improved the present exposition in numerous ways.

\section{}

Let $\field$ be a field of characteristic zero, and let $\powerseries$ be the (associative) $\field$-algebra of formal power series in the noncommutative variables $X,Y$. The variables generate an ideal $\idealXY$ of~$\powerseries$, and a linear subspace $\spanXY$ of~$\idealXY$.

For all $f\in\powerseries$ and $a,b\in\idealXY$, the composite power series $f(a,b)$ is well defined. In particular, we can form $e^a\in\powerseries$ and $\bch ab\in\idealXY$. The associative property $(e^ae^b)e^c=e^a(e^be^c)$ is equivalent to the power-series identity $\bbch{\bch ab}c=\bbch a{\bch bc}$.

We say that two polynomials $p(X,Y)$ and $q(X,Y)$ are \emph{congruent}, and write $p\cong q$, when their difference is a Lie polynomial. The basis of Eichler's proof of the Baker--Campbell--Hausdorff formula is a realisation of the associative property as a congruence equation.

\begin{eichler-congruence}[{\cite[(2)]{eichler}}]
\hypertarget{eichler-congruence}{}
Let $n>2$. Suppose that $\homog[m]XY$ is a Lie polynomial for every $1<m<n$. For $a,b,c\in\spanXY$, we have
\[
\homog ab+\homog{a+b}c\cong\homog bc+\homog a{b+c}.
\]
\end{eichler-congruence}

For the four-term associativity identity $\bbch{\bch ab}{\bch cd}=\bbch a{\bbch{\bch bc}d}$ the analogue of \hyperlink{eichler-congruence}{Eichler's congruence} is as follows.

\begin{lemma}
\label{lemma:quadruple-eichler-congruence}
Assume the hypothesis of \hyperlink{eichler-congruence}{Eichler's congruence}. For $a,b,c,d\in\spanXY$, we have
\begin{equation}
\label{eq:quadruple-eichler-congruence}
\homog ab+\homog cd+\homog{a+b}{c+d}
\cong\homog bc+\homog{b+c}d+\homog a{b+c+d}.
\end{equation}
\end{lemma}

This congruence is proved in the same way as \hyperlink{eichler-congruence}{Eichler's congruence} (see~\ref{sec:bch-formula-recursion}). Its advantage is the relative ease with which it implies the \hyperlink{bch-formula}{Baker--Campbell--Hausdorff formula}.

\begin{lemma}
\label{lemma:commuting-exponentials}
If $a\in\idealXY$, then $\bch aa=2a$ and $\homog aa=0$ for all $n\ge2$.
\end{lemma}

\begin{proof}
The first equality is an easy computation. By homogeneity, we have $\homog aa=c_na^n$ for some $c_n\in\field$. The coefficient $c_n$ is independent of~$a$. In particular $\sum_{n\ge2}c_nX^n=\bch XX-2X=0$. Hence $c_n=0$ for all $n\ge2$.
\end{proof}

\begin{bch-formula}[{\cite[Satz~B]{hausdorff}}]
\hypertarget{bch-formula}{}
For $n\ge2$, the homogeneous term $\homog XY$ is a Lie polynomial.
\end{bch-formula}

\begin{proof}
Direct computation shows that $\homog[2]XY=\oneover2\lie XY$.

Now suppose that $n>2$ and that $\homog[m]XY$ is a Lie polynomial for every $1<m<n$. As a result of Lemma~\ref{lemma:commuting-exponentials}, the congruence of Lemma~\ref{lemma:quadruple-eichler-congruence} becomes
\begin{align}
\label{eq:2X2Y-congruence}
\homog{2X}{2Y}
&\cong\homog XY+\homog{X+Y}Y+\homog X{X+2Y},
&\text{when $X,X,Y,Y$ replaces $a,b,c,d$},\\
\label{eq:2XY-congruence}
2\homog XY
&\cong\homog YX+\homog{X+Y}Y+\homog X{X+2Y},
&\text{when $X,Y,X,Y$ replaces $a,b,c,d$}.
\end{align}
Subtraction gives the congruence
\[
\homog{2X}{2Y}-2\homog XY\cong\homog XY-\homog YX.
\]
By homogeneity, we have $\homog{2X}{2Y}=2^n\homog XY$ and $-\homog YX=(-1)^n\homog XY$, since $-\bch YX=\bch{-X}{-Y}$. Therefore
\begin{equation}
\label{eq:bch-congruence}
\bigl(2^n-(-1)^n-3\bigr)\homog XY\cong0.
\end{equation}
The coefficient is nonzero because $n>2$. Hence $\homog XY\cong0$.
\end{proof}

Note that the proof would be essentially unchanged had we instead considered the alternative four-term associativity identity $\bbch{\bch ab}{\bch cd}=\bbch{\bbch a{\bch bc}}d$, and used the corresponding analogue of \hyperlink{eichler-congruence}{Eichler's congruence}, which reads
\[
\homog ab+\homog cd+\homog{a+b}{c+d}
\cong\homog bc+\homog a{b+c}+\homog{a+b+c}d.
\]

\section{}
\label{sec:bch-formula-recursion}

We shall spell out the proof of Lemma~\ref{lemma:quadruple-eichler-congruence}, according to Eichler's outline \cite[(2)]{eichler}, in order to identify the recursive scheme for $\homog XY$ which is implicit in the above proof of the Baker--Campbell--Hausdorff formula. (Identifying the recursive scheme in Eichler's proof is comparatively tricky.)

For $a,b\in\idealXY$, let
\[
\begin{aligned}
&\deg a=\text{homogeneous term of~$a$ of degree~$n$}
&&(n\ge1),\\
&\homog[<n]ab=\sum_{1<m<n}\homog[m]ab
&&(n>2).
\end{aligned}
\]
Since $\bch ab$ is the sum of its homogeneous terms, and $\deg\homog[m]ab$ is $0$ whenever $m>n$, we have
\begin{equation}
\label{eq:degree-relations}
\begin{aligned}
\deg\bch ab
&=\deg(\homog[1]ab+\homog ab+\homog[<n]ab)\\
&=(\deg a+\deg b)+\homog{\deg[1]a}{\deg[1]b}+\deg\homog[<n]ab.
\end{aligned}
\end{equation}
Note that $\deg\bch ab=\homog ab$ when $a,b$ belong to $\spanXY$, though generally $\deg\bch ab$ is distinct from $\homog ab$.

\begin{proof}[Proof of Lemma~\ref{lemma:quadruple-eichler-congruence}]
We shall express the difference of the two sides of~\eqref{eq:quadruple-eichler-congruence} as a Lie polynomial.

Since $a,b,c,d$ belong to $\spanXY$, an application of~\eqref{eq:degree-relations} gives
\begin{equation}
\label{eq:(ab)(cd)}
\deg\bbch{\bch ab}{\bch cd}
=\homog ab+\homog cd+\homog{a+b}{c+d}+\deg\bhomog[<n]{\bch ab}{\bch cd},
\end{equation}
while two applications of~\eqref{eq:degree-relations} gives
\begin{equation}
\label{eq:a((bc)d)}
\deg\bbch a{\bbch{\bch bc}d}
=\homog bc+\homog{b+c}d+\deg\bhomog[<n]{\bch bc}d
+\homog a{b+c+d}+\deg\bhomog[<n]a{\bbch{\bch bc}d}.
\end{equation}
Since $\bbch{\bch ab}{\bch cd}=\bbch a{\bbch{\bch bc}d}$, subtraction of~\eqref{eq:a((bc)d)} from~\eqref{eq:(ab)(cd)} gives
\begin{equation}
\label{eq:quadruple-eichler-equality}
\begin{aligned}
\MoveEqLeft
\homog ab+\homog cd+\homog{a+b}{c+d}-\homog bc-\homog{b+c}d-\homog a{b+c+d}\\
&=\deg\bhomog[<n]{\bch bc}d
 +\deg\bhomog[<n]a{\bbch{\bch bc}d}
 -\deg\bhomog[<n]{\bch ab}{\bch cd}.
\end{aligned}
\end{equation}
The right side is a Lie polynomial by virtue of the induction hypothesis, since 1) a Lie polynomial of Lie polynomials is itself a Lie polynomial \cite[Satz A]{hausdorff}, and 2) the homogeneous terms of a Lie polynomial are themselves Lie polynomials (evidently).
\end{proof}

\begin{bch-formula}[recursion, I]
\hypertarget{bch-formula-recursion}{}
For $n>2$, we have
\begin{equation}
\label{eq:bch-formula-recursion}
\begin{aligned}
\MoveEqLeft
\bigl(2^n-(-1)^n-3\bigr)\homog XY\\
&=\deg\bhomog[<n]{\bch XY}Y
 -\deg\bhomog[<n]{\bch YX}Y
 +\deg\bhomog[<n]X{\bbch{\bch XY}Y}
 -\deg\bhomog[<n]X{\bbch{\bch YX}Y}.
\end{aligned}
\end{equation}
\end{bch-formula}

\begin{proof}
Use equation~\eqref{eq:quadruple-eichler-equality} to express the congruences~\eqref{eq:2X2Y-congruence} and~\eqref{eq:2XY-congruence} as equalities. It is then apparent that the left side of~\eqref{eq:bch-congruence} equals the right side of~\eqref{eq:bch-formula-recursion}, apart from the terms
\[
-\deg\homog[<n]{2X}{2Y}+\deg\bhomog[<n]{\bch XY}{\bch XY}.
\]
But the first term clearly vanishes, while the second term vanishes by Lemma~\ref{lemma:commuting-exponentials}.
\end{proof}

\section{}

We can simplify the recursive \hyperlink{bch-formula-recursion}{Baker--Campbell--Hausdorff formula}~\eqref{eq:bch-formula-recursion} by observing the following two-fold symmetry.

\begin{lemma}
Let $1\le m\le n$. For $a,b,c,d\in\spanXY$, we have
\begin{align}
\label{eq:symmetry-abc}
&\deg\bhomog[m]{\bch ab}c
=(-1)^{n+m}\deg\bhomog[m]{\bch ba}c,\\
\label{eq:symmetry-abcd}
&\deg\bhomog[m]a{\bbch{\bch bc}d}=
(-1)^{n+m}\deg\bhomog[m]a{\bbch{\bch dc}b}.
\end{align}
\end{lemma}

\begin{proof}
Equation~\eqref{eq:symmetry-abc} is proved by applying homogeneity and associativity in alternation:
\[
(-1)^n\deg\bhomog[m]{\bch ab}c
=\deg\bhomog[m]{\bch{-a}{-b}}{-c}
=\deg\bhomog[m]{-\bch ba}{-c}
=(-1)^m\deg\bhomog[m]{\bch ba}c.
\]
Equation~\eqref{eq:symmetry-abcd} is proved in the same way.
\end{proof}

Now let
\[
\begin{lgathered}
W_{n,m}(X,Y)=\deg\bhomog[m]{\bch XY}Y,\\
W_n(X,Y)=\deg\bbch{\bch XY}Y=\sum_{1\le m\le n}W_{n,m}(X,Y).
\end{lgathered}
\]
The computation of~$\homog XY$ can be reduced to a computation of these polynomials, as shown by the following formula and Lemma~\ref{lemma:bch-formula-even-simplification}.

\begin{bch-formula}[recursion, II]
\hypertarget{bch-formula-recursion-simplified}{}
For $n>2$, we have
\[
\homog XY=\homogeven XY+\homogodd XY,
\]
where
\begin{align}
\label{eq:bch-formula-even}
\bigl(2^n-(-1)^n-3\bigr)\homogeven XY
&=\sum_{\substack{1<m<n\\m\cong n\bmod2}}
\deg\bhomog[m]X{\bbch{\bch XY}Y}
-\deg\bhomog[m]X{\bbch{\bch YX}Y},\\
\label{eq:bch-formula-odd}
\bigl(2^n-(-1)^n-3\bigr)\homogodd XY
&=\sum_{\substack{1<m<n\\m\cong n+1\bmod2}}
2W_{n,m}(X,Y)+(-1)^{n+1}W_{n,m}(2Y,X).
\end{align}
\end{bch-formula}

\begin{proof}
By linearity of~$\deg$, we can write formula~\eqref{eq:bch-formula-recursion} as
\[
\begin{aligned}
\MoveEqLeft
\bigl(2^n-(-1)^n-3\bigr)\homog XY\\
&=\sum_{1<m<n}
 \deg\bhomog[m]{\bch XY}Y
 -\deg\bhomog[m]{\bch YX}Y
 +\deg\bhomog[m]X{\bbch{\bch XY}Y}
 -\deg\bhomog[m]X{\bbch{\bch YX}Y}.
\end{aligned}
\]
It therefore suffices to show that
\begin{enumerate}
\item if $n\cong m\bmod2$, then $\deg\bhomog[m]{\bch XY}Y-\deg\bhomog[m]{\bch YX}Y=0$; and
\item if $n\cong m+1\bmod2$, then $\deg\bhomog[m]X{\bbch{\bch YX}Y}=0$ and
\[
\deg\bhomog[m]{\bch XY}Y
-\deg\bhomog[m]{\bch YX}Y
+\deg\bhomog[m]X{\bbch{\bch XY}Y}
=2W_{n,m}(X,Y)+(-1)^{n+1}W_{n,m}(2Y,X).
\]
\end{enumerate}
The first equality follows from \eqref{eq:symmetry-abc}; the second from \eqref{eq:symmetry-abcd}; the third from \eqref{eq:symmetry-abc}, by virtue of the equality $\bhomog[m]X{\bbch{\bch XY}Y}=\bhomog[m]X{\bch X{2Y}}=(-1)^{m+1}\bhomog[m]{\bch X{2Y}}X$.
\end{proof}

Since $\homog[m]XY$ is a homogeneous Lie polynomial, we can express it as a sum of the form
\begin{equation}
\label{eq:homogeneous-lie-polynomial}
\homog[m]XY=\sum_{0<d<m}L_d(X,Y,\dots,Y)
\quad\text{($d$ occurrences of $Y$)},
\end{equation}
where $L_d(X,Y_1,\dots,Y_d)$ is either zero, if $\homog[m]XY$ has no term of degree~$d$ in~$Y$, or else
a homogeneous Lie polynomial of degree~$m$ in $1+d$ noncommutative variables that is of degree~$1$ in each~$Y_i$. When $L_d(X,Y,\dots,Y)$ is nonzero, its degree in~$Y$ is~$d$ and its degree in~$X$ is~$m-d$.

\begin{lemma}
\label{lemma:bch-formula-even-simplification}
Suppose that $1<m<n$ and $m\cong n\bmod2$. Write $\homog[m]XY$ as in expression~\eqref{eq:homogeneous-lie-polynomial}. We have
\begin{equation}
\label{eq:bch-formula-even-summand}
\begin{aligned}
\MoveEqLeft
\deg\bhomog[m]X{\bbch{\bch XY}Y}
-\deg\bhomog[m]X{\bbch{\bch YX}Y}\\
&=\sum_{0<d<m}
\Biggl(
\sum_{\vec n\in N^0_d}
L_d(X,W_{n_1},\dots,W_{n_d})
+\adjustlimits\sum_{\vec n\in N^1_d}\sum_{\vec m\in M_{\vec n}}
2L_d(X,W_{n_1,m_1},\dots,W_{n_d,m_d})
\Biggr),
\end{aligned}
\end{equation}
where $N^\delta_d$ and $M_{\vec n}$ are the collection of degrees
\[
\begin{lgathered}
{\textstyle
N^\delta_d=\bigl\{(n_1,\dots,n_d)\mid
n_i\ge1,\,\sum_in_i+m-d=n,\,\prod_in_i\cong\delta\bmod2\bigr\}},\\
{\textstyle
M_{\vec n}=\bigl\{(m_1,\dots,m_d)\mid
1\le m_i\le n_i,\,\sum_im_i\cong d+1\bmod2\bigr\}}.
\end{lgathered}
\]
\end{lemma}

\begin{proof}
Let $N_d=\{(n_1,\dots,n_d)\mid n_i\ge1,\,\sum_in_i+m-d=n\}$ and let $W'_\nu=\deg[\nu]\bbch{\bch YX}Y$. Since $\bbch{\bch XY}Y=\sum_{\nu\ge1}W_\nu$ and $\bbch{\bch YX}Y=\sum_{\nu\ge1}W'_\nu$, bilinearity of the commutator yields
\begin{equation}
\label{eq:initial-summation}
\begin{aligned}
\MoveEqLeft
\deg\bhomog[m]X{\bbch{\bch XY}Y}
-\deg\bhomog[m]X{\bbch{\bch YX}Y}\\
&=\sum_{0<d<m}\sum_{\vec n\in N_d}
L_d(X,W_{n_1},\dots,W_{n_d})-L_d(X,W'_{n_1},\dots,W'_{n_d}).
\end{aligned}
\end{equation}
Formula~\eqref{eq:bch-formula-even-summand} is obtained from formula~\eqref{eq:initial-summation} by three successive substitutions:

\paragraph{First substitution.}
Homogeneity and associativity yield
\[
(-1)^\nu\deg[\nu]\bbch{\bch YX}Y
=\deg[\nu]\bbch{\bch{-Y}{-X}}{-Y}
=-\deg[\nu]\bbch{\bch YX}Y,
\]
and hence $W'_\nu=0$ when $\nu\cong0\bmod2$. Thus $L_d(X,W'_{n_1},\dots,W'_{n_d})$ vanishes on~$N^0_d$. Since $N_d$ is the disjoint union of $N^0_d$ and~$N^1_d$, the inner sum of~\eqref{eq:initial-summation} may therefore be substituted with
\begin{equation}
\label{eq:inner-summation}
\sum_{\vec n\in N^0_d}
L_d(X,W_{n_1},\dots,W_{n_d})
+\sum_{\vec n\in N^1_d}
L_d(X,W_{n_1},\dots,W_{n_d})-L_d(X,W'_{n_1},\dots,W'_{n_d}).
\end{equation}

\paragraph{Second substitution.}
Now let $W'_{\nu\mu}=\deg[\nu]\bhomog[\mu]{\bch YX}Y$. We have $W'_\nu=\sum_{\mu\ge1}W'_{\nu\mu}$, and $W'_{\nu\mu}=(-1)^{\nu+\mu}W_{\nu\mu}$ by~\eqref{eq:symmetry-abc}. Therefore the second sum of~\eqref{eq:inner-summation} may be substituted with
\begin{equation}
\label{eq:inner-summation-cancellation}
\adjustlimits\sum_{\vec n\in N^1_d}\sum_{\vec m}
\bigl(1-(-1)^{\varepsilon(\vec n,\vec m)}\bigr)L_d(X,W_{n_1,m_1},\dots,W_{n_d,m_d}),
\end{equation}
where $\vec m=(m_1,\dots,m_d)$, $1\le m_i\le n_i$, and $\varepsilon(\vec n,\vec m)=\sum_i(n_i+m_i)=n-m+d+\sum_im_i$.

\paragraph{Third substitution.}
Finally, since $m\cong n\bmod 2$, the inner sum of~\eqref{eq:inner-summation-cancellation} may be substituted with
\[
\sum_{\vec m\in M_{\vec n}}2L_d(X,W_{n_1,m_1},\dots,W_{n_d,m_d}).
\]

Altogether these substitutions yield formula~\eqref{eq:bch-formula-even-summand}.
\end{proof}

\section{}

To illustrate the recursive \hyperlink{bch-formula-recursion-simplified}{Baker--Campbell--Hausdorff formula (II)}, we shall compute $\homog XY=\homogeven XY+\homogodd XY$ for $n=3,4,5$, according to formulas~\eqref{eq:bch-formula-even}, \eqref{eq:bch-formula-odd}, and \eqref{eq:bch-formula-even-summand}, cf.~\cite[(26)]{hausdorff}. Explicit Bernoulli numbers are notably absent, cf.~\cite[\S I.3]{hausdorff}.

For the sake of simplicity we write $Z_n$ for $\homog XY$, and $\ilie{abc\dots d}$ for $[[\dots[[a,b],c],\dots],d]$.

\subsection*{Case $n=3$.}
We have
\[
\begin{lgathered}
W_{3,2}(X,Y)=\homog[2]{Z_2}Y=\oneover4\ilie{XY^2},\\
6\homogodd[3]XY
=2W_{3,2}(X,Y)+W_{3,2}(2Y,X)
=\oneover2\ilie{XY^2}+\oneover2\ilie{YX^2}.
\end{lgathered}
\]
Summation~\eqref{eq:bch-formula-even} is vacuous when $n=3$, hence
\[
\homog[3]XY=\homogodd[3]XY=\oneover{12}\ilie{XY^2}+\oneover{12}\ilie{YX^2}.
\]

\subsection*{Case $n=4$.}
We have
\begin{align*}
&\begin{aligned}
12\homogeven[4]XY
&=\deg[4]\bhomog[2]X{\bbch{\bch XY}Y}-\deg[4]\bhomog[2]X{\bbch{\bch YX}Y}\\
&=\ilie{XW_{3,2}}\\
&=\oneover4\ilie{YXYX},
\end{aligned}
\shortintertext{and}
&W_{4,3}(X,Y)
=\oneover{12}\ilie{Z_2Y^2}
+\oneover{12}\ilie{YZ_1Z_2}
+\oneover{12}\ilie{YZ_2Z_1}
=\oneover{24}\ilie{YXYX},\\
&12\homogodd[4]XY
=2W_{4,3}(X,Y)-W_{4,3}(2Y,X)
=\oneover4\ilie{YXYX}.
\end{align*}
Hence
\[
\homog[4]XY=\oneover{24}\ilie{YXYX}.
\]

\subsection*{Case $n=5$.}
We have
\begin{align*}
&W_{1,1}=X+2Y,\\
&W_2=Z_2+\homog[2]{Z_1}Y=\ilie{XY},
\quad\text{(by~\eqref{eq:degree-relations})}\\
&\begin{aligned}
30\homogeven[5]XY
&=\deg[5]\bhomog[3]X{\bbch{\bch XY}Y}
 -\deg[5]\bhomog[3]X{\bbch{\bch YX}Y}\\
&=\oneover6\ilie{W_{3,2}X^2}
 +\oneover{12}\ilie{XW_2^2}
 +\oneover6\ilie{XW_{1,1}W_{3,2}}
 +\oneover6\ilie{XW_{3,2}W_{1,1}}\\
&=\oneover{12}\ilie{YX^3Y}
 +\oneover{12}\ilie{XYXYX}
 +\oneover6\ilie{YXYXY}
 +\oneover{12}\ilie{XY^3X},
\end{aligned}
\shortintertext{and}
&W_{5,2}(X,Y)
=\homog[2]{Z_4}Y
=\oneover{48}\ilie{YXYXY},\\
&W_{5,4}(X,Y)
=\oneover{24}\ilie{YZ_1YZ_2}+\oneover{24}\ilie{YZ_2YZ_1}
=\oneover{48}\ilie{YXYXY}-\oneover{48}\ilie{XY^4},\\
&\begin{aligned}
30\homogodd[5]XY
&=2W_{5,2}(X,Y)+W_{5,2}(2Y,X)
 +2W_{5,4}(X,Y)+W_{5,4}(2Y,X)\\
&=\oneover{12}\ilie{YXYXY}
 +\oneover6\ilie{XYXYX}
 -\oneover{24}\ilie{XY^4}
 -\oneover{24}\ilie{YX^4}.
\end{aligned}
\end{align*}
Hence
\[
\homog[5]XY
=\oneover{120}\ilie{YXYXY}
+\oneover{120}\ilie{XYXYX}
+\oneover{360}\ilie{YX^3Y}
+\oneover{360}\ilie{XY^3X}
-\oneover{720}\ilie{XY^4}
-\oneover{720}\ilie{YX^4}.
\]

\bibliographystyle{plain}

\end{document}